\documentclass[reqno,centertags,11pt]{amsart}
\usepackage{amsmath,amsthm,amscd,amssymb, epsfig}
\usepackage{latexsym}
\usepackage{mathabx}  

\usepackage[colorlinks=true]{hyperref}

\usepackage {verbatim}
\usepackage{enumerate, verbatim}
\def\ch{ \marginpar{change} }
\newcommand{\Hmm}[1]{\leavevmode{\marginpar{\tiny%
$\hbox to 0mm{\hspace*{-0.5mm}$\leftarrow$\hss}%
\vcenter{\vrule depth 0.1mm height 0.1mm width \the\marginparwidth}%
\hbox to 0mm{\hss$\rightarrow$\hspace*{-0.5mm}}$\\\relax\raggedright #1}}}

\newcommand{\be}{\begin{eqnarray}}
\newcommand{\ee}{\end{eqnarray}}

\newcommand{\bn}{\begin{eqnarray*}}       
\newcommand{\en}{\end{eqnarray*}}
\newtheorem{theorem}{Theorem}
\newtheorem{definition}{Definition}
\newtheorem{lemma}{Lemma}
\newtheorem{remark}{Remark}

\newcommand{\bea}{\begin{eqnarray*}}
\newcommand{\eea}{\end{eqnarray*}}
\newcommand{\ben}{\begin{eqnarray}}
\newcommand{\een}{\end{eqnarray}}
\newcommand{\beq}{\begin{equation}}
\newcommand{\eeq}{\end{equation}}

\newcommand{\R}{\ensuremath{\mathbb{R}}}

\newcommand{\la}{\lambda}

\newcommand{\half}{\frac{1}{2}}

\newcommand{\eps}{\varepsilon}
\newcommand{\de}{\delta}

\renewcommand{\ch}{\mathbf{1}}

\newcommand{\om}{\omega}

\newcommand{\ka}{\kappa}

\begin{document}

\title{Singular solutions with vorticity control for a nonlocal system of evolution equations}

\date{\today}


\author{Vu Hoang         \and Maria Radosz}

\address{Rice University, 
Department of Mathematics-MS 136, Box 1892, Houston, TX 77251-1892
}
\email{Vu.Hoang@rice.edu}
\email{maria\_radosz@hotmail.com}


\begin{abstract}
We investigate a system of nonlocal transport equations in one spatial dimension. The system can be regarded as a model for the 3D Euler equations in the hyperbolic flow scenario.
We construct blowup solutions with control up to the blowup time.
\end{abstract}

\maketitle

\section{Introduction}
In fluid mechanics, one-dimensional model problems that capture various aspects of the three-dimensional Euler equations have a long-standing history. Recently, various one-dimensional equations \cite{HouLuo2,sixAuthors,CKY} have been proposed that are models for 3D Euler axisymmetric flow and are connected to the hyperbolic flow scenario proposed by T. Hou and G. Luo in \cite{HouLuo1}. 

In this paper, we study the following two regularized version of one-dimensional system proposed in \cite{HouLuo2}: the equations are given by
\beq\label{model}
\begin{split}
\om_t + u \om_x = \frac{\rho}{x},\\
\rho_t + u \rho_x = 0.
\end{split}
\eeq
The velocity field is given by
\begin{align}
u(x, t) = - x \int_x^\infty \frac{\om(y, t)}{y}~dy.
\label{u2}
\end{align}
Our model \eqref{model} is posed on the half-line $\R^+$ and can be regarded as a simpler version of the equations proposed in \cite{CKY}. We will always consider smooth initial data $\om_0, \rho_0$ that with compact suport in $(0, \infty)$. Smooth solutions exist for short times, but may blow up in finite time.

To show finite-time blowup is not hard, and we do not consider this to be the main problem. Instead, we are motivated by a much deeper and more important question:\noindent 
\begin{center}
\emph{ What is the intrinsic blowup mechanism and how does finite time blowup happen?} 
\end{center}
In particular, how does the vorticity profile behave close to the blowup time? Due to the nonlocal nature of the velocity field, this is a challenging question. 

Similar questions were asked for an equation proposed by C\'ordoba-C\'ordoba-Fontelos \cite{CCF} in \cite{silvestre2014transport}, where it was conjectured that generically cusps form in the solution.

In this paper, we develop a systematic approach to constructing blowup solutions with control over the vorticity up to the time of blowup. Moreover, we make statements about the asymptotic behavior of the vorticity close to the blowup time. The significance of system \eqref{model} is that it allows us to explore the essential issues in a more manageable context.  

As a basic observation note that \eqref{model} has a family of explicit singular solutions on $(0, \infty)$:
\beq\label{inf}
\om(x, t) = k x^{-\frac{1}{2}}, ~~\rho(x, t) = k^2
\eeq
with $k > 0$. 
This is seen by computing the velocity field using \eqref{u2}, yielding
\bea
u(x, t) &= - k x^{\frac{1}{2}}.
\eea

It is straightforward to verify that \eqref{inf} solves
\eqref{model} on $(0, \infty)$. The starting point for our work comes from the following natural question: What happens if we start with initial data $\om_0$ that looks qualitatively like in Figure \ref{fig1}? One can imagine modifying the stationary singular profile $x^{-1/2}$ in a suitable way as to make it compactly supported in $(0, \infty)$, making a ``cut-off" for small $A_0 > 0$ and for large $x$. 
$\rho_0$ will also be adjusted in a suitable way so that $\rho_0$ has compact support in $(0, \infty)$.

Our main results show that all such initial vorticity configurations $\om_0$ lead to finite-time blowup. The precise conditions on $\om_0, \rho_0$ are given below.
Moreover, if $\rho_0$ is appropriately chosen, we will obtain bounds that indicate the asymptotic behavior $\om \sim x^{-1/2}$ at the blowup time. 

We close the introduction by remarking that the exponent in the singular profile \eqref{inf} is model-dependent. In \cite{HouLuo2}, a 1D model of the 3D axisymmetric Euler equations was proposed:  
\beq\label{model_HL}
\begin{split}
\om_t + u \om_x = \rho_x,\\
\rho_t + u \rho_x = 0,
\end{split}
\eeq
where the velocity is given by $u_x = H \om$, $H$ being the Hilbert transform. The analogue of \eqref{inf} for \eqref{model_HL} is
\begin{align}\label{inf2}
\begin{split}
\om(x, t) = k |x|^{-1/3}\operatorname{sgn}(x), \\
\rho(x, t) = c_1 k^2 |x|^{1/3} + c_2 k^3 t
\end{split}
\end{align}
where $c_1, c_2>0$ are suitable constants.  
We plan to explore the connection between \eqref{inf2} and finite-time singularities in future work.

\begin{figure}
\includegraphics[scale=0.8]{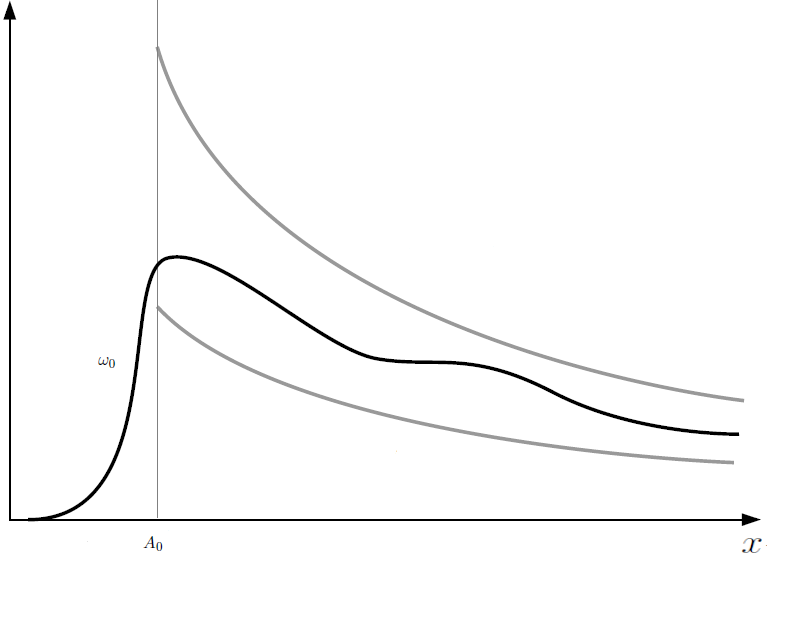}
\caption{Smooth initial vorticity $\om_0$ with compact support in $(0,\infty)$. Large $x$ not shown. } \label{fig1}
\end{figure}

\section{Main results}
We first make a few general remarks about our problem setup. 

\begin{definition}\label{soln2}
Let $(\om_0, \rho_0)\in C_0^1(0, \infty)\times C_0^1(0, \infty)$ be given functions. A pair of
functions $(\om, \rho)$ satisfying
\begin{align}
\begin{split}
\om\in C([0, T], C_0^1(0, \infty))\cap C^1([0, T], C(0, \infty)),\\
\rho \in C^1([0, T], C_0^1(0, \infty))\cap C^1([0, T], C(0, \infty))
\end{split}
\end{align}
is solution of \eqref{model}+\eqref{u2} if \eqref{model} holds in the
classical sense on $[0, T]\times(0, \infty)$.
\end{definition}

\begin{theorem}\label{t3}
Let initial data $(\om_0, \rho_0)\in C_0^1(0, \infty)\times C_0^1(0, \infty)$ be given. Then there exists a $T>0$ and
a unique pair of functions $(\om, \rho)$
solving \eqref{model1}+\eqref{u3} with $\om(\cdot, 0)=\om_0(\cdot), \rho(\cdot, 0) = \rho_0(\cdot)$. Moreover, for nonnegative initial data, the following holds: If there exists a $\de>0$ such that
\begin{align}
\limsup_{t\to T} \int_0^t \|\om(\cdot, s)\|_{L^\infty(0, \de)}~ds < \infty,
\end{align}
then the solution can be continued past $t = T$.
\end{theorem}
The proof of Theorem \ref{t3} uses the standard particle-trajectory method for local existence and is very similar to results in \cite{Hoang2D}, hence the proof will be omitted. 

For the following controlled blowup Theorem, we consider a more
general system than \eqref{model}:
\beq\label{model1}
\begin{split}
\om_t + u \om_x = \frac{\rho}{x^\beta},\\
\rho_t + u \rho_x = 0,
\end{split}
\eeq
with velocity field given by
\begin{align}
u(x, t) = - x \int_x^\infty \frac{\om(y, t)}{y}~dy.
\label{u3}
\end{align}
The parameter $\beta$ is an arbritrary, positive number $\beta > 0$.

\begin{theorem}\label{t4}
There exist constants $A_0 > 0$ and positive $p, q, \phi, \psi$ with
\begin{align}
    0< p < q, ~~p+q = \beta 
\end{align}
and the following properties: There exists a nonempty class of initial data, called suitably prepared initial data in the sense of Definition \ref{def_suitablyPrep1} and characterized by $A_0, p, q, \phi, \psi$ such that for any such initial data $\om_0, \rho_0\in C_0^1(0, \infty)$ the smooth solution $(\om, \rho)$ of \eqref{model}+\eqref{u2}
is defined on $[0, T_s)$ for some finite $T_s>0$ and cannot be continued
as a smooth solution past $T_s$. Moreover,  
\begin{align}\label{vort_blowup}
\lim_{t\to T_s} \int_0^t \|\om(\cdot, s)\|_{L^\infty(0, \infty)}~ds = \infty
\end{align}
holds. For these solutions, we have the following controls 
on the vorticity: 
\beq
 \phi x^{-p} \leq \om(x, t) \leq \psi x^{-q} \quad (t \in [0, T_s))
\eeq
for all $x$ such that $A(t) \le x \le 1$. Additionally, there exists a time $T_1 > 0$ determined by the initial data such that if the solution stays smooth on $[0, T_1)$,
\begin{align}
\lim_{t\to T_1} A(t) = 0
\end{align}
holds. Here, $A(t)$ is the particle trajectory with $A(0) =A_0, \dot A(t) = u(A(t), t)$.
\end{theorem}
\begin{remark}
$A_0$ is represented in Figure \ref{fig1}. Suitably prepared initial data in the sense of Definition \ref{def_suitablyPrep1} essentially requires that $\om$ lies above a function of the form $\phi x^{-p}$ and below the function $\psi x^{-q}$ for $x\in [A_0, 1]$, where $A_0$ has to be chosen sufficiently small. 
\end{remark}

The following theorem gives bounds for the vorticity that are precise in the exponent. For simplicity, the following Theorem is stated for solutions of \eqref{model} with velocity field \eqref{u2}, but a 
similar version exists for \eqref{model1}. 
\begin{theorem}\label{main_theorem_2}
There exists a choice of parameters such that the set of suitably prepared initial data in the sense of Definition \ref{def_suitablyPrep2} is not empty. For all such initial data, the lifetime $T_s$ of the corresponding solutions is finite and \eqref{vort_blowup} holds. The estimate
\begin{align}\label{asymptotic_omega}
\phi x^{-1/2} \le \om(x, t) \le \psi x^{-1/2} 
\end{align}
holds for some positive $\phi, \psi, \la_0$ and for all $x\in (A(t), \la_0]$. Additionally, there exists a time $T_1 > 0$ such that if
the solution stays smooth on $[0, T_1)$, we have
\begin{align}
\lim_{t \to T_1} A(t) = 0.
\end{align}
\end{theorem}

\begin{remark}
\begin{enumerate}
\item[(a)]
Theorem \ref{main_theorem_2} shows that there exists a class of initial data that blows up in finite time in such a way that $\om(x, t)$ has the same asymptotic behavior as the stationary singular profile of the type given in \eqref{inf} as $t\to T_s$, with the correct exponent $-\frac{1}{2}$. Given any initial profile of the general shape in Figure \eqref{fig1}, there are two possibilities: either the solution blows up before the particle $A(t)$ reaches zero, or at the moment when $A(t)$ reaches zero. In the second case, we will have the estimate
\begin{align}\label{asymptotic_omega_2}
\phi x^{-1/2} \le \lim_{t\to T_1}\om(x, t) \le \psi x^{-1/2} 
\end{align}
for all $x\in (0, \la_0]$.
\item[(b)] In order to obtain Theorem \ref{main_theorem_2}, we have to make more stringent assumptions on the class of initial compared to Theorem \ref{t4}. In Theorem \ref{t4}, we can also allow variations in $\rho_0(x)$, whereas in \ref{main_theorem_2}, we assume that $\rho_0(x) = 1$ identically on an interval.
\end{enumerate}
\end{remark}

The plan of the paper is as follows: In Section \ref{sec:blowupSmooth} we show controlled blowup for the slightly more general model \eqref{model1} with \eqref{u3}, proving bounds for the vorticity that involve inverse power functions with differing exponents. In \ref{sec:blowupAsymp}, we improve the exponents to get the values predicted by the stationary singular profile. 

\section{Controlled blowup for smooth solutions}\label{sec:blowupSmooth}
In this section, we consider smooth solutions for \eqref{model1}. We write $u(x, t) = -x Q(x, t)$ where
\begin{align}\label{def_Q2}
Q(x, t) = \int_x^\infty \frac{\om(y, t)}{y}~dy. 
\end{align}

\begin{definition}\label{def_suitablyPrep1}
Given parameters $p, q, \phi, \psi, \de, m, A_0, \eps > 0$ 
with
\begin{align}
\begin{split}\label{restr_p}
0 < \eps < A_0 < 1, ~~p, ~~q\in (0, 1),\\
p + q = 1+\beta-1, \\
0 < p < \frac{\beta}{2},\\
\phi < \psi-\de,
\end{split}
\end{align}
we say that the initial data $(\om_0, \rho_0)$ is suitably prepared if the following hold
\begin{align}
\begin{split}\label{cond_init}
\om_0, \rho_0~~\text{are nonnegative on}~\R^+,\\
 \operatorname{supp}\om_0 \subset [\eps, 4],~~ \operatorname{supp}\rho_0 \subset [\eps, 3],\\ 
 \om_0(x) < \psi x^{-q} \quad (A_0 \leq x\leq 1)\\
 \om_0(x) < \psi-\delta  \quad (1 \leq x \leq 4)\\
 \phi x^{-p} < \om_0(x) \quad (A_0 \leq x \leq 1)\\
 \phi  < \om_0(x) \quad (1\leq x\leq 3)\\
  \rho_0 \leq m \quad (x\in \R^+)\\
m^{-1} \leq \rho_0(x)\quad (A_0 \leq x \leq 2).
\end{split}
\end{align}
\end{definition}
An important preliminary observation is as follows: Since $\om_0, \rho_0$ are nonnegative, $\om(\cdot, t), \rho(\cdot, t)$ remain nonnegative for all times and as a consequence, $u(x, t)\leq 0$ for all $x\geq 0$ and $t\in [0, T_s)$, where $T_s$ is the lifetime of the solution. In other words, all particles move to the left. The specific conditions on $\om_0, \rho_0$ is are written down in a way that is  convenient for us. In the same way, we can treat  essentially all initial data with compact support in $(0,\infty)$ that has the same general qualitative behavior.  
\begin{definition} 
Given the positive numbers $A_0, \phi, \psi, p, q$  
satisfying \eqref{restr_p}, we say that a smooth solution $(\om, \rho)$ of \eqref{model}+\eqref{u3} is \emph{controlled} on $[0, T)$ if $\om(\cdot,0), \rho(\cdot, 0)$ are suitably prepared with the same
parameter values and if
\begin{align}\label{cc15}
\begin{split}
\om(x, t) &< \psi x^{-q}\quad (x\in [A(t), 1]),\\
\om(x, t) &< \psi \quad (x\in [1, 3]),\\
\phi x^{-p} &<  \om(x, t) \quad (x\in [A(t), 1]),\\
\phi &< \om(x, t)\quad (x\in [1, 2])
\end{split}
\end{align}
hold for all $t\in [0, T)$. Here, $A(t)$ is the particle trajectory that satisfies $A(0) = A_0$.
\end{definition}
Note that $\om(x, t) = 0$ for $x \geq 4$ and all $t\in [0, T_s)$. Another consequence of \eqref{cond_init} and the fact that all particles are moving to the left is that 
\begin{align}\label{add_bound}
\om(x, t) \leq \psi \quad (x\in [3, 4]).
\end{align}
This follows because $\rho(x, t) = 0$ for all $x\geq 3$.

The following Lemma gives an important estimate on $Q$, provided $\om$ is controlled.
\begin{lemma}\label{controlQ1}
Suppose $(\om, \rho)$ is controlled on $[0, T)$. Then
\begin{align}
\begin{split}
b_0 \phi x^{-p} \leq Q(x, t) \leq  b_1 \psi x^{-q}\quad (x\in [A(t), 1]).
\end{split}
\end{align}
Here, $b_0,  b_1$ are given by
\begin{align}\label{def_b}
b_0 = \frac{\min\{1, p\log(2)\}}{p}, \quad
b_1 = \frac{\max\{1, q\log(4)\}}{q}. 
\end{align}
Moreover, we have the estimate
\begin{align}\label{add_up_Q}
Q(x, t) \leq \psi \log\left(\frac{4}{x}\right) 
\end{align}
for all $x\geq 1$.
\end{lemma}
\begin{proof}
First we look at the lower bound. For $x\in [A(t), 1]$, we use the third and fourth line of 
\eqref{cc15} to compute
\begin{align*}
Q(x, t) &\geq \phi \int_x^{1} y^{-p - 1} dy + \phi \int_1^2 y^{-1}~dy\\
&= \frac{\phi}{p}\left(x^{-p}-1\right) + \phi \log(2)\\
& = \frac{\phi x^{-p}}{p}\left((1-x^{p})\cdot 1 + x^{p} p\log(2)\right)\\
& \geq \frac{\phi x^{-p}}{p} \min\{1, p\log(2)\}.
\end{align*}
In going from the third to the fourth line of the calculation, we have used that $x^{p} \in [0, 1]$ so that the bracket is a convex combination of the numbers $1$ and $p\log2$. This shows the lower bound for $Q$.
The upper bound is found by a similar calculation, by integrating the first inequality of 
\eqref{cc15} and using in addition \eqref{add_bound}. The estimate \eqref{add_up_Q} follows by integrating $\om(x, t) \leq \psi$ for $x\in [1, 4]$.
\end{proof}

Next we fix $\phi, \psi, p, q, A_0$ as in the following Lemma: 
\begin{lemma}\label{lem6}  
Given $m\ge 1$, there exist $\phi, \psi, p, q, A_0, \de$ such that the following hold true:
\begin{align}\label{cond_params}
\begin{split}
0< p< q < \beta \leq 1,\quad p+q=\beta\\
\frac{m}{ b_0 q}<\phi\psi<\frac{1}{m  b_1 p}\\
\frac{A_0^{p}}{\phi  b_0 p} < \frac{\log(3/2)}{\psi \log(4)}\\
\frac{m A_0^{p}}{\phi  b_0 p} < \de
\end{split}
\end{align}
The set of all initial data $(\om_0, \rho_0)$ such that
\eqref{cond_init} holds is nonempty.
\end{lemma}
\begin{proof}
We choose a small $p > 0$. Define $q$ by $q = \beta - p$ and as a consequence, $q$ converges to $\beta$ as $p\to 0$. Note that $b_0 \to \log 2$ as $p\to 0$ and that $b_1$ has a finite limit as $p\to 0$. We hence see that
$$
\frac{m}{b_0 q} < \frac{1}{m b_1 p}
$$
for sufficiently small $p > 0$. Now choose a $c \in 
[\frac{m}{b_0 q}, \frac{1}{m b_1 p}]$ and define $\phi = c \psi^{-1}$. It is clear that by choosing $\psi$ large and $\delta$ small, the set of $\om_0$ satisfying \eqref{cond_init} is non-empty. At this point, we regard the values of $p, q, \phi, \psi$ to be fixed. Finally, the third, fourth and fifth condition of \eqref{cond_params} hold if
$A_0$ is sufficiently small. 
\end{proof}

The following Lemma gives an upper bound on the lifetime of 
a controlled solution:
\begin{lemma}\label{upperboundTime}
Suppose the smooth solution $(\om, \rho)$ is controlled on $[0, T)$. Then $T < T^{*}$, where
\begin{align}
T^{*} = \frac{A_0^{p}}{\phi  b_0 p}.
\end{align}
\end{lemma}
\begin{proof}
From Lemma \ref{controlQ1}, we obtain the estimate
\begin{align}
- u(A(t), t) \geq \phi  b_0 A(t)^{1-p}.
\end{align}
The conclusion follows by integrating the inequality $-\dot A \geq \phi  b_0 A(t)^{1-p}$. 
\end{proof}

The following Lemma plays a key role. It allows us to propagate the control condition.
\begin{lemma}[Trapping lemma]\label{lem_trap1}
Assume that $\om$ is controlled on $[0, T)$ for some positive $T> 0$, $T < T_s$. Consider particle trajectories $X(t)$ such that $A_0\leq X(0)$. Let $t^* \leq t$ be some time such that $X(t^*) \leq 1$ and such that
\begin{align}
\phi X(t^*)^{-q} < \om(X(t^*), t^*) < \psi X(t^*)^{-q}
\end{align}
holds. Assume that \eqref{cond_params} holds and that $m^{-1} \leq \rho(X(t), t)\leq m$. Then 
\begin{align}
\phi X(T)^{-p} < \om(X(T), T) < \psi X(T)^{-q}.
\end{align}
\end{lemma}
\begin{proof}
We integrate along $X(t)$ to obtain for all $t\leq T$:
\begin{align*}
\om(X(t), t) &= \om(X(t^*), t^*) + \rho_0(X_0)\int_{t^*}^t\frac{ds}{X(s)^{\beta}} \\
& \geq  \om(X(t^*), t^*) + m^{-1}\int_{t^*}^t\frac{1}{X(s)^{\beta}}\frac{(-\dot X(s))}{(-X(s))} ds\\
& \geq  \om(X(t^*), t^*) + \frac{1}{m \psi b_1}\int_{t^*}^t\frac{1}{X(s)^{\beta}}\frac{(-\dot X(s))}{X(s)^{1+q}} ds\\
& \geq \om(X(t^*), t^*) + \frac{1}{m \phi  b_1 (\beta-q )}(X(t)^{-\beta+q}-X(t^{*})^{-\beta+q})\\
&\geq \om(X(t^*), t^*) + \frac{1}{m \psi b_1 p}(X(t)^{-p}-X(t^{*})^{-p})\\
&\geq \left[\om(X(t^*), t^*)X(t)^{p} + \frac{1}{m \psi b_1 p}(1-X(t)^{p}X(t^{*})^{-p})\right] X(t)^{-p}\\
&\geq \left[\om(X(t^*), t^*) X(t^*)^p \eta  + \frac{1}{m \psi b_1 p}(1-\eta)\right] X(t)^{-p}
\end{align*} 
where we have used Lemma \ref{controlQ1}, used the relation $p+q = \beta$ and have written $\eta = X(t)^p X(t^*)^{-p}$. Now note that $\eta \in [0, 1]$, so in order for $\phi X(t)^{-p} < \om(X(t), t)$ to hold, the following two conditions 
are sufficient:
\begin{align*}
\om(X(t^*), t^*) X(t^*)^p > \phi,\\
\frac{1}{m \psi b_1 p} > \phi.
\end{align*} 
The first of those holds by assumption and the second condition holds because of \eqref{cond_params}. The upper bound on $\om(X(T), T)$ is obtained in a similar way. 
\end{proof}

\begin{theorem}\label{blowupSmooth}
Let the parameters $p, q, \phi, \psi, A_0, \delta$ be chosen as in the conclusion of Lemma \ref{lem6}, and let the initial data $(\om_0, \rho_0)$ be suitably prepared with parameter values satisfying \eqref{cond_params}. Let $(\om, \rho)$ be the unique smooth 
solution with initial data $(\om_0, \rho_0)$ defined on its maximal existence interval $[0, T_s)$
characterized by the property that
$\lim_{t\to T_s} \|\om\|_{L^\infty(\R^+)} = \infty$.
Then $(\om, \rho)$ is controlled on $[0, T_s)$. As a consequence, $T_s < \infty$. 
\end{theorem}
\begin{proof}
Because \eqref{cond_init} holds,
 there exists a time $\tau>0$ such that the solution
is controlled on the time interval $[0, \tau)$. So assume that the solution is controlled up to a time $T_c>0$ and that at $t=T_c$, one of the strict inequalities in 
\eqref{cc15} is violated. This means that there is a particle trajectory $X(t)$ such that at time $t=T_c$, we have equality in one the bounds in \eqref{cc15}. First we consider the case that one of the lower bounds is violated. 

\emph{Case 1}: $X(T_c)\in [1, 2]$. Note that $X(0) \geq 1$.
Consider the following subcases, the first one being $X(0) \in [1, 3]$. Because $\rho_0\geq 0$, it follows by integrating \eqref{model1} along the trajectory that 
\begin{align*}
\om(X(T_c), T_c) \geq \om(X(0)) > \phi
\end{align*}
because of $\om(X(0), 0) = \om_0(X(0)) > \phi$.
(see \eqref{cond_init}). We now show that $X(0) \geq 3$ cannot occur.
From the assumption that the solution was controlled up to the time $T_c$ and Lemma
\ref{controlQ1}, we obtain the estimate $-\dot X(t)\leq  \psi\log(4) X(t)$ and hence
\begin{align*}
X(t) \geq X(0) e^{-\psi \log (4) t}.  
\end{align*}
Noting that $t\leq T^*$ by Lemma \ref{upperboundTime} and using the explicit form of $T^*$, we obtain, by the third condition of \eqref{cond_params} the bound 
$X(T_c) > 2$, provided $X(0) \geq 3$. This means that the particles with $X(0) \geq 3$ did not have enough time to reach positions inside $[1, 2]$. 

\emph{Case 2}: $X(T_c)\in [A(T_c), 1]$. Here we distinguish the again the two subcases
$X(0)\in [A_0, 1]$ and $X(0)\in [1, 3]$. $X(0) \geq 3$ cannot occur, as shown above.
Consider first the case that $X(0)\in [1, 3]$. We then know that $m^{-1} \leq \rho(X(t), t) \leq m$, so we can apply Lemma \ref{lem_trap1} where $t^*$ is the time such that
$X(t^*) = 1$ to see that 
\begin{align}
\phi X(T_c)^{-p} < \om(X(T_c), T_c) 
\end{align}
holds. In the case $X(0)\in [A_0, 1]$, we take $t^*=0$ and again apply Lemma \ref{lem_trap1}.
In summary, the lower bounds in \eqref{cc15} do not fail. 

For the upper bounds, we also distinguish the two cases:

\emph{Case 1}: $X(T_c)\in [1, 3]$. The third condition in \eqref{cond_params} ensures that
$X(0) < 4$. Since $X(s) \geq 1$ for $s\leq T_c$, we have the estimate
\begin{align*}
\om(X(T_c), T_c) \leq \om_0(X(0)) + m T_c.
\end{align*}
Inserting the fourth condition of \eqref{cond_params}, 
using $\om_0(X(0)) < \psi - \de$ and $T_c \leq T^*$ (from Lemma \ref{upperboundTime}) shows that 
\begin{align*}
\om(X(T_c), T_c) < \psi - \de + m T^* 
\end{align*}
which is less than $\psi$ because of the fifth inequality in \eqref{cond_params}. 

\emph{Case 2}: $X(T_c)\in [A(T_c), 1]$. This case is similar to case 2 above for the lower bound, using the trapping Lemma \eqref{lem_trap1}.

From all the preceding it follows that the solution remains controlled on its whole existence interval $[0, T_s)$.

It remains to show that the lifetime $T_s$ of the solution is finite. From the fact that the solution is controlled on the whole interval
$[0, T_s)$ and Lemma \eqref{controlQ1}, it follows that
$$
- A(t) \geq  b_0 A(t)^{1 - p},
$$ 
and since $1 - p < 1$, this implies that $A(t)$ reaches zero in finite time. 
\end{proof}

%
\section{Asymptotic estimates for the vorticity}\label{sec:blowupAsymp}
%

In this section we show Theorem \ref{main_theorem_2}. The central idea consists of iterating the previous barrier construction on smaller and smaller scales. The exponents $p, q$ are replaced by sequences $p_n, q_n$ and the quality of the barriers will improve as the scales get smaller. For convenience, we do the construction for $\beta = 1$.

We therefore introduce a number of sequences satisfying certain conditions. Let $\la_n$ be a strictly decreasing sequence such that
$$1>\la_{-2}>\la_{-1}>\la_0>\la_n\to 0\quad (n\to\infty).$$
Let
$$I_n=[\la_{n},\la_{n-1}].$$
$\phi_n,\psi_n,q_n,p_n, m_n,M_n$ are sequences with the following properties:
\begin{align}\label{condition1}
0 < \phi_n < \psi_n, \\
p_n\nearrow \half,\quad q_n\searrow \half,\quad p_n+q_n=1\\
m_n\nearrow 1,\quad M_n\searrow 1.
\end{align}
For the sequences the following conditions shall hold:
\begin{align}
\phi_n \psi_n &= 1~~~(n\geq 1)\label{c__1}\\
\la_n/\la_{n-1} &\to 0~~~~(n\to \infty)\label{c_1}\\
\la_n^{q_n-q_{n-1}}&\longrightarrow 1~~(n\to\infty)\label{c_0}\\
\frac{q_n}{q_{n-1}}\frac{\psi_{n-1}}{\psi_n}\la_{n-1}^{q_n-q_{n-1}}&\le 1~~~~(n\geq 2)\label{c1}\\
\frac{p_n}{p_{n-1}} \frac{\phi_{n-1}}{\phi_n}\la^{p_n-p_{n-1}} &\ge 1~~~~(n\geq 2) \label{c2}\\
\phi_1 \la_0^{-p_1} &< \psi_1 \la_0^{-q_1} \label{relay_start}\\
\phi_{n-1}\la_{n-1}^{-p_{n-1}}&\ge\phi_{n}\la_{n-1}^{-p_{n}}\label{relay1}~~~~(n\geq 2)\\
\psi_{n-1}\la_{n-1}^{-q_{n-1}}&\le\psi_{n}\la_{n-1}^{-q_{n}}~~~~(n\geq 2)\label{relay2}\\
\frac{p_n}{m_n(1-p_n)}<1 &<\frac{q_n}{M_n(1-q_n)}~~~~ (n\geq 1)\label{trapping}
\end{align}
The conditions \eqref{relay1} and \eqref{relay2} are called relay conditions, \eqref{trapping} is called trapping condition. Note that
\begin{align*}
\phi_n x^{-p_n} < \psi_n x^{-q_n} \quad (x\in I_n).
\end{align*}

Again, the class of initial data for which our statements in Theorem \ref{main_theorem_2} holds is contained in the following Definition. 
\begin{definition}
The initial data $(\om_0, \rho_0)$ is called suitably prepared if for some $0 < A_0 < \la_0$ and some $\de>0$,  
$\om_0$ and $\rho_0$ satisfy
\begin{align}\label{con2}
\begin{split}
\om_0(x) \geq 0, \rho_0(x) \geq 0 &\quad (x\in \R^+)\\
\operatorname{supp} \om_0, \operatorname{supp} \rho_0 \subset [\eps, \lambda_{-2}+\de]~~&\text{for some}~0 < \eps < A_0\\
\phi_n x^{-p_n} < \om_0(x) < \psi_nx^{-q_n}&\quad (x\in I_n,~x\ge A_0,~ n=1,2,3,\ldots)\\
\om_0(x)<\psi_1\la_0^{-q_1}-\de &\quad(x\in[\la_0,\la_{-2}+\de])\\
\om_0(x)>\phi_1\la_0^{-q}&\quad(x\in[\la_0,\la_{-1}])\\
\rho_0(x) = 1 &\quad (x\in [A_0, \la_{-2}])
\end{split}
\end{align}
where $\de>0$ is so small such that $\phi_1 \la_0^{-p_1} < \psi_1\la_0^{-q_1}-\de$ holds.
\end{definition}

\begin{definition}\label{def_suitablyPrep2}
Let $(\om, \rho)$ be a smooth solution. We say that the solution is controlled on $[0, T)$, if for all $t\in [0, T)$
the following conditions  hold:
\begin{align}\label{con1}
\begin{split}
\phi_nx^{-p_n} < \om(x,t) < \psi_nx^{-q_n}&\quad (x\in I_n,~x\ge A(t),~ n=1,2,3,\ldots)\\
\om(x,t)<\psi_1\la_0^{-q_1}&\quad(x\in[\la_0,\la_{-2}])\\
\om(x,t)>\phi_1\la_0^{-q}&\quad(x\in[\la_0,\la_{-1}])
\end{split}
\end{align}
Here, $A(t)$ is a particle trajectory such that $A(0) = A_0$, $0 < A_0 < \la_{0}$, where $A_0$
is fixed.
\end{definition}

Theorem \ref{main_theorem_2} is implied by the following: 
\begin{theorem}
There exists a choice of parameters such that the set suitably prepared initial data in the sense of Definition 
is not empty. For all such initial data, the lifetime $T_s$ of the corresponding solutions is finite and for all particle 
trajectories $X(t)$ with 
\begin{align*}
A_0\leq X(0)\leq \lambda_0,
\end{align*} 
the estimate
\begin{align}\label{asymptotic_omega_1}
\phi_{\infty} \leq \lim_{t\to T_s} \om(X(t), t) X(t)^{1/2} \leq \psi_{\infty}
\end{align}
holds for some $\phi_\infty, \psi_\infty>0$ with $\phi_\infty \psi_\infty = 1$.
Moreover, there exists a positive time $T^* < \infty$ such that if the solution $(\om(\cdot, t), \rho(\cdot, t))$ remains smooth on the time interval $[0, T^*)$, then 
\begin{align}
\lim_{t\to T^*} A(t) = 0.
\end{align}
\end{theorem} 

A crucial role is played by the following Lemma, which is used to control the velocity.
\begin{lemma}\label{lem_control_Q}
Let $\om$ satisfy the control conditions \eqref{con1} on $[0,T]$. Define
\begin{align}
F = \sup_{n\geq 2} \left[\frac{q_n \psi_{n-2}}{q_{n-2}\psi_{n}}\la_{n-1}^{q_n - q_{n-1}} \la_{n-1}^{q_{n-1}-q_{n-2}} + \frac{p_n \phi_{n-1}}{p_{n-1}\phi_n}\la_{n-1}^{p_n-p_{n-1}}+\frac{q_n \psi_1}{\psi_{n}}\right].
\end{align}
Then for all $n\ge 1$
\begin{align}
\frac{\phi_n}{p_n}m_nx^{-p_n}&\le Q(x,t)\le \frac{\psi_n}{q_n}M_nx^{-q_n} 
\end{align}
where $m_n =1-\mu_n, M_n = 1+\mu_n$ with 
\begin{align}
\mu_n &= C F \left(\frac{\la_{n-1}}{\la_{n-2}}\right)^{q_1}+ C \log\left(\frac{\la_{-2}}{\la_0}\right)\la_{0}^{-q_1} \la_{n-1}^{q_n} \quad (n\geq 2)\\
\mu_{1} &= \log\frac{\la_{-2}}{\la_0}
\end{align}
and $C>0$ is a universal constant. 
\end{lemma}

\begin{proof}We start with the upper bound.
Set $f_{n}(y)=\sum_{j=1}^n\ch_{I_j}(y)\psi_j y^{-q_j}$. As a preparation, we will estimate
\begin{align}
\int_{\la_{n-1}}^{\la_0}\frac{f_{n-1}(y)}{y}~dy.
\end{align}
for $n\geq 2$. Setting $y=\la_{n-1}z$,
\begin{align*}
\int_{\la_{n-1}}^{\la_0}\frac{f_{n-1}(y)}{y}~dy
&=\int_1^{\la_0/\la_{n-1}}f_{n-1}(\la_{n-1}z)z^{-1}~dz\\
&=\int_1^{\la_0/\la_{n-1}}\sum_{j=1}^{n-1}\ch_{[\frac{\la_j}{\la_{n-1}},\frac{\la_{j-1}}{\la_{n-1}}]}(z)\psi_jz^{-q_{j}-1}\la_{n-1}^{-q_j}~dz\\
&\le\int_1^{\infty}\psi_{n-1} z^{-1-q_{n-1}}\la_{n-1}^{-q_{n-1}}~dz+\int_{\la_{n-2}/\la_{n-1}}^{\la_0/\la_{n-1}}\sum_{j=1}^{n-2}\ch_{[\frac{\la_j}{\la_{n-1}},\frac{\la_{j-1}}{\la_{n-1}}]}(z)\psi_jz^{-1-q_{j}}\la_{n-1}^{-q_j}~dz\\
&\le \frac{\psi_{n-1}}{q_{n-1}} \la_{n-1}^{-q_{n-1}}+\int_{\la_{n-2}/\la_{n-1}}^{\la_0/\la_{n-1}}\sum_{j=1}^{n-2}\ch_{[\frac{\la_j}{\la_{n-1}},\frac{\la_{j-1}}{\la_{n-1}}]}(z)\psi_jz^{-1-q_{j}}\la_{n-1}^{-q_j}~dz
\end{align*}
For the integral in the last line we have the following estimate.
\begin{align*}
\int_{\la_{n-2}/\la_{n-1}}^{\la_0/\la_{n-1}}\sum_{j=1}^{n-2}\ch_{[\frac{\la_j}{\la_{n-1}},\frac{\la_{j-1}}{\la_{n-1}}]}(z)\psi_jz^{-1-q_{j}}\la_{n-1}^{-q_j}~dz&=\sum_{j=1}^{n-2}\la_{n-1}^{-q_j}\int_{\la_{n-2}/\la_{n-1}}^{\la_0/\la_{n-1}}\ch_{[\frac{\la_j}{\la_{n-1}},\frac{\la_{j-1}}{\la_{n-1}}]}(z)\psi_jz^{-1-q_{j}}~dz
\\
&=\sum_{j=1}^{n-2}\frac{\psi_j}{q_j}(\la_j^{-q_{j}}-\la_{j-1}^{-q_j})
\\
&\le \frac{\psi_{n-2}}{q_{n-2}}\la_{n-2}^{-q_{n-2}}-\frac{\psi_{1}}{q_{1}}\la_{0}^{-q_1}
\end{align*}
In the last step we have used the condition \eqref{c1}.
Hence
\begin{align}
\frac{q_n}{\psi_n}\la_{n-1}^{q_n}\int_{\la_{n-1}}^{\la_0} \frac{f_{n-1}}{y}~d y &\le \frac{q_n}{\psi_n}\la_{n-1}^{q_n} \left[ \frac{\psi_{n-1}}{q_{n-1}} \la_{n-1}^{-q_{n-1}}+ \frac{\psi_{n-2}}{q_{n-2}}\la_{n-2}^{-q_{n-2}} \right]\\
&\le 1+ \frac{q_n \psi_{n-2}}{\psi_n q_{n-2}} \la_{n-1}^{q_n} \la_{n-2}^{-q_{n-2}}\\
&= 1+ \frac{q_n \psi_{n-2}}{\psi_n q_{n-2}} \la_{n-1}^{q_n} \la_{n-1}^{-q_{n-1}} \la_{n-1}^{q_{n-1}} \la_{n-1}^{-q_{n-2}} \la_{n-1}^{q_{n-2}} \la_{n-2}^{-q_{n-2}}\\
&\le 1+ F \left(\frac{\la_{n-1}}{ \la_{n-2}}\right)^{1/2}\\ 
&\le 1 + F \left(\frac{\la_{n-1}}{ \la_{n-2}}\right)^{q_1}
\end{align}
using \eqref{c1} again and $C$ is a universal constant. 
From the control conditions \eqref{con1}, we obtain for $x\in I_n, x\geq A(t)$ 
\begin{align}
Q(x, t) &\le \frac{\psi_n}{q_n}(x^{-q_n} - \la_{n-1}^{-q_n}) + \int_{\la_{n-1}}^{\la_0} \frac{f_{n-1}(y)}{y}~dy + \psi_1 \la_0^{-q_1} \log\frac{\la_{-2}}{\la_0}\\
 &\le \frac{\psi_n}{q_n}(x^{-q_n} - \la_{n-1}^{-q_n}) + \frac{\psi_n}{q_n}\la_{n-1}^{-q_n}\left(1+F\la_{n-1}^{q_1}\la_{n-2}^{-q_1}\right) + \psi_1 \la_0^{-q_1} \log\frac{\la_{-2}}{\la_0}\\
 &= \frac{\psi_n}{q_n}x^{-q_n}\left[1 - x^{q_n}\la_{n-1}^{-q_n} + x^{q_n}\la_{n-1}^{-q_n}\left(1+F\la_{n-1}^{q_1}\la_{n-2}^{-q_1}\right) +  \frac{q_n}{\psi_n} x^{q_n}\psi_1 \la_0^{-q_1} \log\frac{\la_{-2}}{\la_0}\right]
\end{align}
Writing $\ka = x^{q_n}\la_{n-1}^{-q_n}$ and noting that $\ka \in [0, 1]$, we see that
the square bracket is convex combination of $1$ and the number
\begin{align}
\left(1+F\la_{n-1}^{q_1}\la_{n-2}^{-q_1}\right) +  \frac{q_n}{\psi_n}\la_{n-1}^{q_n}\psi_1 \la_0^{-q_1} \log\frac{\la_{-2}}{\la_0}
\end{align}
which is less than
\begin{align}
\left(1+F\la_{n-1}^{q_1}\la_{n-2}^{-q_1}\right) + F\log\left(\frac{\la_{-2}}{\la_0}\right)\la_{0}^{-q_1} \la_{n-1}^{q_n}.
\end{align} This yields the desired upper bound for $Q$ in the case $n\geq 2$.
The computation for $n = 1$ is very similar. 

Now we turn to the lower bound. Set $g_{n}(y)=\sum_{j=1}^n\ch_{I_j}(y)\phi_j y^{-p _j}$.  Then setting $y=\la_{n-1}z$
\begin{align*}
\int_{\la_{n-1}}^{\la_0}\frac{g_{n-1}(y)}{y}~dy
&=\int_1^{\la_0/\la_{n-1}}g_{n-1}(\la_{n-1}z)z^{-1}~dz\\
&=\int_1^{\la_0/\la_{n-1}}\sum_{j=1}^{n-1}\ch_{[\frac{\la_j}{\la_{n-1}},\frac{\la_{j-1}}{\la_{n-1}}]}(z)\phi_jz^{-p_{j}-1}\la_{n-1}^{-p_j}~dz\\
&\ge\int_1^{\la_{n-2}/\la_{n-1}}\phi_{n-1}z^{-p_{n-1}- 1}\la_{n-1}^{-p_{n-1}}~dz\\
&=\frac{\phi_{n-1}}{p_{n-1}}\la_{n-1}^{-p_{n-1}}\left(1-\la_{n-2}^{-p_{n-1}}\la_{n-1}^{-p_{n-1}}\right)\\
\end{align*}

For $x\in I_n, x\geq A(t)$ and $n\geq 2$,
\begin{align}
Q(x, t) &\ge \frac{\phi_n}{p_n}(x^{-p_n} - \la_{n-1}^{-p_n}) + \frac{\phi_{n-1}}{p_{n-1}} (\la_{n-1}^{-p_{n-1}} - \la_{n-2}^{-p_{n-1}})\\
&\ge \frac{\phi_n}{p_n}x^{-p_n} \left[1 - x^{p_n} \la_{n-1}^{-p_{n}} + \frac{\phi_{n-1}p_n}{\phi_{n}p_{n-1}} x^{p_n} (\la_{n-1}^{-p_{n-1}} - \la_{n-2}^{p_{n-1}})
\right]
\end{align}
Writing again $\ka = x^{q_n}\la_{n-1}^{-q_n}$ and noting that $\ka \in [0, 1]$, we see that
the square bracket is convex combination of $1$ and the number
\begin{align}
\frac{\phi_{n-1}p_n}{\phi_{n}p_{n-1}} \la_{n-1}^{p_n} (\la_{n-1}^{-p_{n-1}} - \la_{n-2}^{p_{n-1}})
\end{align}
which has the lower bound
\begin{align*}
& 1 - \frac{\phi_{n-1}p_n}{\phi_{n}p_{n-1}} \la_{n-1}^{p_n}\la_{n-2}^{-p_{n-1}}\\
&\ge 1 - \frac{\phi_{n-1}p_n}{\phi_{n}p_{n-1}} \la_{n-1}^{p_n}\la_{n-2}^{-p_{n-1}} \\
&\ge 1 - \frac{\phi_{n-1}p_n}{\phi_{n}p_{n-1}} \la_{n-1}^{p_n} \la_{n-1}^{-p_{n-1}} \la_{n-1}^{p_{n-1}}\la_{n-2}^{-p_{n-1}}\\
&\ge 1 - F\left(\frac{\la_{n-1}}{\la_{n-2}}\right)^{q_1}.
\end{align*}
This implies the desired lower bound on $Q$ for $n\geq 2$. 
\end{proof}

\begin{lemma}\label{choose_seq}
Let $0 < \la_{-2} < 1$ and numbers $\phi_1, \psi_1$ with
$\phi_1 \psi_1 = 1$ be given. The numbers $\la_0, \la_{-1}$ and the
the sequences $p_n, q_n, \phi_n, \psi_n, \la_n$ can be chosen in such a way such that \eqref{condition1} to \eqref{trapping} hold and such that moreover 
\begin{align*}
\lim_{n\to \infty} \phi_n > 0.
\end{align*}
As a consequence, the set of initial suitably prepared initial data is not empty for
all $0 < A_0 < \la_0$.
\end{lemma}
\begin{proof}
Recall that in Lemma \ref{lem_control_Q}, $m_n$ and $M_n$ were defined and were written as $m_n = 1 - \mu_n, M_n = 1+\mu_n$. We make the following ansatz ($n\geq 1$):
\begin{align}\label{ansatz}
q_n = \frac{1}{2} + \eps_n, ~p_n = \frac{1}{2} - \eps_n
\end{align}
with a positive, montonotically decreasing sequence $\eps_n$ such that 
$\eps_n \to 0$ as $n\to \infty$. 

First we pick a small $\eps_1 > 0$ and set $\eps_n = \eps_1 e^{-(n-1)}$ and $\la_n = \la_0 e^{-L n^2}$ for $n\geq 1$ and $L > 0$ to be fixed later. $\la_0$ and $\la_{-1}$ will also be fixed later.
\eqref{c_1} and \eqref{c_0} are satisfied.
For $n \geq 2$, we define the sequences $\phi_n, \psi_n$ by
\begin{align}\label{def_seq_phi}
\phi_n = \phi_1 \prod_{j=2}^{n} \la_{j-1}^{\eps_{j-1}-\eps_j}, ~~~~~
\psi_n = \psi_1 \prod_{j=2}^{n} \la_{j-1}^{-\eps_{j-1}+\eps_j}
\end{align}
In this way, \eqref{c__1} and \eqref{relay1} as well as \eqref{relay2} hold. Inserting \eqref{def_seq_phi} into \eqref{c1} reduces \eqref{c1} to $q_n \leq q_{n-1}$, which holds since $q_n$ is monotone decreasing. \eqref{c2} is verified in the same way. 
We note that as a consequence of our choice of $\eps_n, \lambda_n$, $\lim_{n\to \infty} \phi_n$ is positive, since 
\begin{align*}
\sum_{j\geq 2} |\log \la_{j-1}||\eps_{j-1}-\eps_j| < \infty
\end{align*}
for each $L > 0$.

After using \eqref{ansatz} and $m_n = 1 - \mu_n, M_n = 1+\mu_n$, we see that a sufficient condition for the two inequalities in \eqref{trapping} is 
\begin{align}\label{cond_trap_1}
\mu_n + 2 \mu_n \eps_n < 4 \eps_n.
\end{align}
In case $n=1$, $\mu_1 = \log(\la_{-2}/\la_{0})$ and \eqref{cond_trap_1} can be satisfied by choosing $\la_0$ to be sufficiently close to $\la_{-2}$. Observe now that $F$ from Lemma \ref{lem_control_Q} can be bounded in terms of the parameters $\eps_1, \la_0$, uniformly in the parameter $L \geq 1$. From the definition of $\mu_n$ for $n\geq 2$, it follows
that
\begin{align*}
\mu_n \le C F e^{-L(4n - 3)} + C \log(\la_{-2}/\la_{0}) \la_0^{-\eps_1} e^{-\frac{1}{2} L n^2}
\end{align*}
so that \eqref{cond_trap_1} holds if $L$ is chosen sufficiently large (dependent on $\eps_1$). 
It is obvious that the set of suitably prepared initial data is nonempty for all $0 < A_0 < \la_0$, since we have $\phi_n x^{-p_n} < \psi_n x^{-q_n}$ for all $x\in I_n$, $n\geq 1$.
\end{proof}

Again, the following trapping Lemma plays a key role.
\begin{lemma}[Trapping lemma]\label{lem_trap}
Assume that $\om$ is controlled on $[0, T)$ for some positive $T> 0$, $T < T_s$. Let $n\geq 1$ and consider particle trajectories $X(t)$ such that $ A(t)\leq X(t)\leq \lambda_{n-1}$.  Suppose $t^* \leq t$ is some time such that $X(t^*) \leq \lambda_{n-1}$ and such that
\begin{align}
\phi_n X(t^*)^{-p_n} < \om(X(t^*), t^*) < \psi_n X(t^*)^{-q_n}
\end{align}
holds. Assume that \eqref{trapping} holds and that $\rho(X(t), t)=1$. Then 
\begin{align}
\phi_n X(T)^{-p} < \om(X(T), T) < \psi_n X(T)^{-q}
\end{align}
\end{lemma}
The proof is very similar to the proof of Lemma \ref{lem_trap1} and uses \eqref{trapping}.

\begin{theorem}
Let the numbers $0 < \la_{-2} < 1, \phi_1, \psi_1$ with $\phi_1\psi_1 = 1$ be given and let the sequences $\phi_n, \psi_n, \la_n$ be chosen as in Lemma \ref{choose_seq}.
Suppose $A_0 > 0$ is sufficiently small. Then any smooth solution with suitably prepared initial data stays controlled on the time interval $[0, T_s)$ and $T_s < \infty$. 
\end{theorem}
\begin{proof}
As in the proof of Theorem \ref{blowupSmooth}, $A_0$ can be chosen a-priori so small that by the time the solution blows up, the second and third inequalities of \eqref{con1} still hold. We therefore focus on the persistence of the bounds 
\begin{align}\label{proof_eq_1}
\phi_n x^{-p_n} < \om(x,t) < \psi_nx^{-q_n}&\quad (x\in I_n~~\text{for some}~n\geq 1,~x\ge A(t)).
\end{align}
Since the initial data is suitably prepared in the sense of Definition \ref{def_suitablyPrep2}, there exists a small time interval $[0, \tau)$ on which the smooth solution is controlled. Let $0 < T_c \le T_s$ be the maximal time on which the solution is controlled. If $T_c < T_s$, then there is a particle trajectory such
that for some $n \geq 1$, $X(T_c)\in I_n$ and either 
\begin{align}\label{proof_eq_2}
\phi_n X(T_c)^{-p_n} = \om(X(T_c), T_c)~~\text{or}~~\psi_n X(T_c)^{-q_n} = \om(X(T_c), T_c).
\end{align}
If $X(0) \in I_n$, we let $t^* = 0$, otherwise we let $t^*$ be such that $X(t^*) = \la_{n-1}$ and apply the trapping Lemma \ref{lem_trap} and conclude that \eqref{proof_eq_2} does not occur. To show that $T_s < \infty$ works as in the proof of Lemma \ref{blowupSmooth}. 
\end{proof}

Now we can finish the proof of Theorem \ref{main_theorem_2}. Let all parameters be chosen as in the previous Theorem. We hence have 
\begin{align*}
\om(x, t) \geq \phi_n x^{-p_n} \geq \phi_n x^{-p_n+1/2} x^{-1/2}
\end{align*}
for all $x\in [A(t), \la_0]$, where $n$ is such that $x\in I_n$.
Suppose now that $T^*$ is such that $A(T^*) = 0$.
Note that $x^{-p_n+1/2} \geq \la_{n}^{-p_n+1/2}\to 1$ as $n\to \infty$. This follows from the explicit construction of the sequences $\la_n, p_n$ in Lemma \ref{choose_seq}. From the same Lemma we also know $\lim_{n\to \infty} \phi_n > 0$.
Hence we have the estimate
\begin{align*}
\om(x, t) \geq \phi x^{-1/2}
\end{align*}
with some suitable $\phi > 0$. The upper bound $\om(x, t)\leq \psi x^{-1/2}$ follows from similar considerations. 
%
\section{Acknowledgements}
%

We would like to thank Serguei Denissov, who in a discussion, made VH aware of the role an exact singular solution may play in the construction of smooth blowup solutions. Yao Yao suggested a generalization of the problem we originally considered. VH also expresses his gratitude to the German Research Foundation (DFG): many ideas contained in this work were developed when VH was supported by grants FOR 5156/1-1 and FOR 5156/1-2. Finally, VH would like to acknowledge support by NSF grant DMS-1614797.


\begin{thebibliography}{10}
\bibitem{sixAuthors}
K.~Choi, T.Y.~Hou, A.~Kiselev, G.~Luo, V.~\v{S}ver\'{a}k and Y.~Yao,
\newblock On the finite-time blowup of a 1d model for the 3d axisymmetric {Euler} equations.
\newblock {\em arXiv:1407.4776}, 2014.

\bibitem{CKY}
K.~Choi, A.~Kiselev and Y.~Yao,
\newblock Finite time blow up for a 1d model of 2d {Boussinesq} system.
\newblock {\em Comm. Math. Phys.}, 334(3):1667--1679, 2015.

\bibitem{Castro}
A.~Castro and D.~C\'ordoba, Infinite energy solutions of the surface quasi-geostrophic equation, 
Advances in Mathematics 225 (2010) 1820–1829.

\bibitem{ConstRev}
P.~Constantin:\emph{On the Euler equations of incompressible fluids}, Bull. Amer. Math. Soc, Volume 44, Number 4, October (2007), Pages 603–621.

\bibitem{CLM}
P.~Constantin, P.D.~Lax and A.~Majda,
\newblock A simple one-dimensional model for the three-dimensional vorticity equation,
\newblock {\em Comm. Pure Appl. Math.}, 38:715--724, 1985.

\bibitem{constantin1994formation}
P.~Constantin, A.~Majda and E.~Tabak
\newblock Formation of strong fronts in the 2-d quasigeostrophic thermal active
  scalar,
\newblock {\em Nonlinearity}, 7(6):1495--1533, 1994.

\bibitem{CCF}
A.~C\'ordoba, D.~C\'ordoba and M.A.~Fontelos,
\newblock Formation of singularities for a transport equation with nonlocal velocity,
\newblock {\em Ann. of Math.(2)}, 162(3):1377--1389, 2005.

\bibitem{Majda}
A. Madja, A. Bertozzi,
\newblock Vorticity and Incompressible Flow,
\newblock Cambridge University Press, 2002.

\bibitem{Hoang2D} V. Hoang, B. Orcan-Ekmeckci, M. Radosz and H. Yang, Blowup with vorticity control for a 2D model of the Boussinesq equations, Preprint {\tt https://arxiv.org/abs/1608.01285}.

\bibitem{HouLuo1}
T.Y.~Hou and G.~Luo,
\newblock Toward the finite-time blowup of the 3d axisymmetric {Euler} equations: A numerical investigation,
\newblock {\em Multiscale Model. Simul.}, 12(4):1722--1776, 2014.

\bibitem{HouLuo2}T.Y.~Hou and G.~Luo,
\newblock Potentially singular solutions of the 3D axisymmetric Euler equations.
\newblock {\em PNAS}, vol. 111 no. 36, 12968-12973,
\emph{DOI 10.1073/pnas.1405238111.}

\bibitem{HouLi} T.Y.~Hou and C.~Li, Dynamic stability of the three-dimensional axisymmetric Navier-Stokes equations
with swirl, Commun. Pure Appl. Math, vol. LXI, 661--697 (2008)

\bibitem{Pengfei}
T.Y.~Hou and P.~Liu.
\newblock Self-similar singularity of a 1D model for the 3D axisymmetric Euler equations.
\newblock Research in the Mathematical Sciences (2015) 2:5
DOI 10.1186/s40687-015-0021-1.

\bibitem{kiselev2013small}
A.~Kiselev and V.~\v{S}ver\'{a}k.
\newblock Small scale creation for solutions of the incompressible two
  dimensional {Euler} equation.
\newblock {\em Ann. of Math.(2)}, 180(3):1205--1220, 2014.

\bibitem{Saffmann} 
P.G.~Saffman, Vortex Dynamics. Cambridge University Press, Cambridge (1992)

\bibitem{silvestre2014transport}
L.~Silvestre and V.~Vicol,
\newblock On a transport equation with nonlocal drift.
\newblock Trans. Amer. Math. Soc., DOI: http://dx.doi.org/10.1090/tran6651.

\end{thebibliography}
\end{document}